\documentclass[a4paper]{article}
\usepackage{graphicx}
\usepackage[utf8]{inputenc}
\usepackage{amsmath}
\usepackage{amssymb}
\usepackage{amsfonts}
\usepackage{fleqn}
\usepackage{psfrag}
\usepackage{mathrsfs}
\usepackage{amssymb}
\usepackage{amsthm}
\usepackage{algorithm}
\usepackage{algpseudocode}
\usepackage[parfill]{parskip}
 \usepackage{xr-hyper}
   \usepackage{hyperref}
\usepackage{pifont}
\usepackage[toc,page]{appendix} 
\makeatletter
\def\BState{\State\hskip-\ALG@thistlm}
\makeatother

\usepackage{float}
\newfloat{algorithm}{H}{lop}

\usepackage{pifont}
 \newtheorem{theorem}{Theorem}[section]
 \newtheorem{lemma}[theorem]{Lemma}

 \theoremstyle{definition}
\newtheorem{definition}{Definition}[section]
\newcommand{\vertiii}[1]{{\left\vert\kern-0.25ex\left\vert\kern-0.25ex\left\vert #1 
    \right\vert\kern-0.25ex\right\vert\kern-0.25ex\right\vert}}
\begin{document}
\title{A priori error estimation for elastohydrodynamic lubrication using interior-exterior penalty approach}
\date{}
\author{Peeyush Singh$^{1}$ \\
 \small \noindent $^1$Vellore Institute of Technology-AP, University\\
  \small \noindent Department of Mathematics, Andhra Pradesh-522237, India
 \\  
  \small \noindent $^1$E-mail: peeyush.singh@vitap.ac.in\\ 
  }
\maketitle
\thispagestyle{empty}
\begin{abstract}
In the present study, an interior-exterior penalty discontinuous Galerkin finite element method (DG-FEM) is analysed for solving Elastohydrodynamic lubrication (EHL) line and point contact problems. The existence of discrete penalized solution is examined using Brouwer's fixed point theorem. Furthermore, the uniqueness of solution is proved using Lipschitz continuity of the discrete solution map under light load parameter assumptions.
A priori error estimates are achieved in $L^{2}$ and $H^{1}$ norms which are shown to be optimal in mesh size $h$ and suboptimal in polynomial degree $p$.
The validity of theoretical findings are confirmed through series of  numerical experiments.  
\end{abstract}
\vspace{0.2in}
\newpage
\maketitle
\section{Introduction}
Discontinuous Galerkin finite element methods (DG-FEMs) are now widely used in scientific computation to achieve better accuracy and their ﬂexibility in handling nonuniform degrees of approximation as well as flexibility in local mesh adaptivity. While DG-FEMs have been profoundly shown to be very successful, the theory ensuring the convergence of the algorithm and the 
advantages over Elastohydrodynamic Lubrication (EHL) line and point contacts problems still under development. Recently, several results
have been obtained for DG-FEMs for strongly nonlinear elliptic partial differential equations (see for example \cite{gudi07,gudi08}).

This paper presents the mathematical analysis of DG-FEMs  for EHL problems using interior-exterior penalty approach.
However, the idea is more generic and it can be easily extended to more general variational inequality problems too.
In this study, author demonstrates priori error estimate in the underline norm, prove existence and uniqueness of the discrete DG-FEM formulation of the EHL problem and find the optimal and suboptimal rate of convergence under discussed norm.
Recently, a nearby method to approximate the EHL solution is to solve the analogue discrete inequality using discontinuous Galerkin finite volume method (DG-FVM) approach discussed (see for example \cite{peeyush2020}) which is restricted under lower regularity assumptioms of penalized soultion of EHL point contact problem.\\
Exterior penalty methods has a significant role in from of conceptual perspective. One striaght way justification is, it convert inequalities to equation.
A priori estimates for finite element methods for elliptic obstacle problems were proved in \cite{Scholz} using regularization technique imposing the unilateral constraint approximately through a penalty term depending on a regularization
parameter $\epsilon$ and relating the mesh size $h$ of the finite element mesh to the regularization parameter $\epsilon$.
By motivated by the same route, in this article author prove priori error estimate of discontinuous Galerkin-finite element methods using interior-exterior penality approach.
\subsection{Continuous EHL Model Problems}
\subsubsection{Line contact model}
Two cylinder rolling in the positive $x$-direction seperated with lubricant (oil, liquid etc) are modelled in the form of
variational inequality as
\begin{figure}
\centering
\includegraphics[width=4.0in, height=3.0in, angle=0]{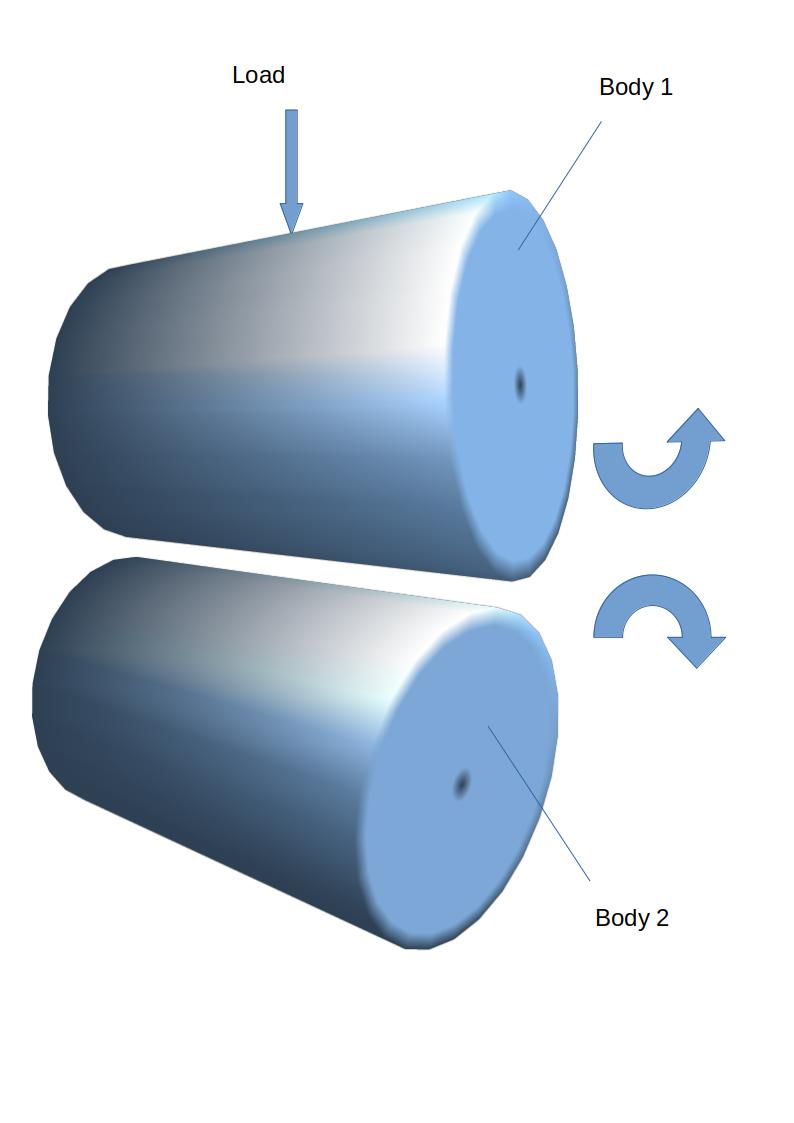}
\caption{Line contact schematic diagram}
\label{fig:defm}
\centering
\includegraphics[width=4.0in, height=3.in, angle=0]{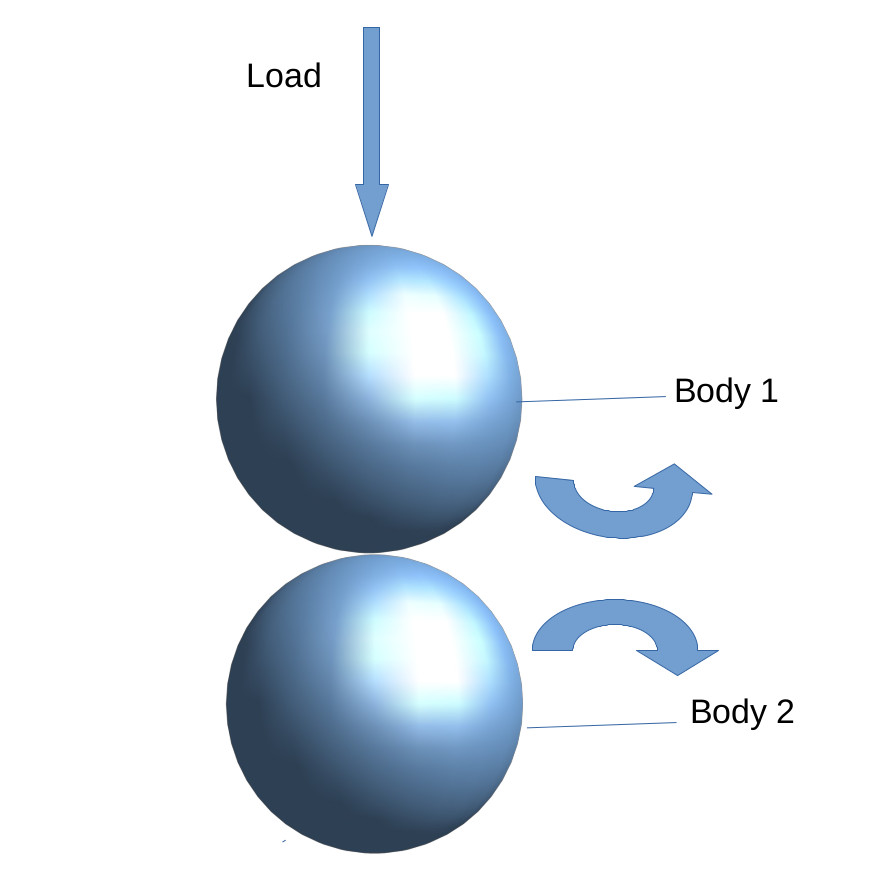}
\caption{Point contact schematic diagram}
\label{fig:undefm}
\end{figure}
\begin{align}\label{eq:1}
 \frac{\partial }{\partial x} \Big(\epsilon^{*} \frac {\partial u}{\partial x}\Big)
 \le \frac {\partial (\rho h)}{\partial x}
\end{align}
\begin{align}\label{eq:2}
 u\ge 0
\end{align}
\begin{align}\label{eq:3}
 u.\Big[\frac{\partial }{\partial x} \Big(\epsilon^{*} \frac {\partial u}{\partial x}\Big)
 -\frac {\partial (\rho h)}{\partial x}\Big] = 0,
\end{align}
 where 
  \begin{align*}
   \epsilon = \dfrac{\bar{\rho} h^{3}}{\bar{\eta} \lambda_{\text{line}}}.
  \end{align*}
  $u$ and $h$ are the dimensionless pressure and film thickness, $\bar{\rho}(u)$ and 
  $\bar{\eta}(u)$ are dimensionless density and viscosity, and $\lambda_{\text{line}}$ is a dimensionless speed 
  parameter:
  \begin{align*}
   \lambda_{\text{line}}= \dfrac{6\eta_{0}v_{s}R^{2}}{b^{3}p_{H}},
  \end{align*}
where $\eta_{0}=0.04$ (ambient pressure viscosity), $v_{s}=v_{1}+v_{2}$ (sum of velocity ), $p_{H}=\frac{Eb}{4R_{x}}$ (maximum Herizian pressure),
$R_{x}=0.02$ (reduced radius of curvature ) and $b=\frac{4.0R_{x}}{\sqrt{W/(2.0\pi)}}$ (half width Hertizian contact).
$G_{0}=3500$ (material parameter), $U=7.3\times 10^{-11}$ (dimensionless speed parameter), $W=1.3\times 10^{-4}$ (dimensionless load parameter),
$h_{00}=0.0000015042$, $\alpha=1.59 \times 10^{-8}$, $ E=G_{0}/\alpha$, $z=\dfrac{\alpha}{\Big(5.1\times 10^{-9}(\log \eta_{0} + 9.67)\Big)}$,
$\lambda=\frac{(12ER_{x}^{3}U)}{(b^{3}p_{H})}$, $U=\frac{(\eta_{0}u_{s})}{(2ER_{x})}$.\\
The nondimensionless viscosity $\bar{\eta}$ is defined according to 
\begin{align}
 \bar{\eta}(u) = e^{\Bigg\{ \Bigg( \dfrac{\alpha p_{0}}{z}  \Bigg) 
 \Bigg(-1+\Big(1+\dfrac{up_{H}}{p_{0}}\Big)^{z}   \Bigg)   \Bigg\} },
\end{align}
where $p_{0}=1.98 \times 10^{-8}$.
Dimensionless density $\bar{\rho}$ is given by
\begin{align}
 \bar{\rho}(u) = \dfrac{0.59 \times 10^{9} + 1.34 u p_{H}}{0.59 \times 10^{9} + u p_{H}}
\end{align} 
The nondimensionalized film thickness equation can be written as
  \begin{align}
 h(x) =h_{00}+\frac{{x}^2}{2}-\frac{1}{\pi}\int_{-\infty}^{\infty} u(x'){\ln|x-x'| }dx' 
\end{align}
where $h_{00}$ is a constant.\\
Dimensionless force balance equation is read as
 \begin{align}
 \int_{-\infty}^{\infty} u(x')dx' -\dfrac{\pi}{2}=0
\end{align}
Define Derichlet boundary condition by taking sufficently large bounded domain as
 \begin{align}\label{eq:4}
  u= 0 \quad \text{on} \quad \partial \Omega.
 \end{align}
The film thickness equation is in dimensionless form is written as follows
\small
\begin{equation}\label{eq:5}
h_{d}(x) = h_{0}+\frac{x^{2}}{2}-\frac{1}{\pi}\int_{\Omega}\log|x-x^{'}|u(x^{'})dx^{'},
\end{equation}
\normalsize
where $h_{0}$ is an integration constant.\\
The dimensionless force balance equation is defined as follows
\begin{align}\label{eq:6}
 \int_{-\infty}^{\infty}u(x') dx' = \frac{\pi}{2}
\end{align}
\subsubsection{Point contact model}
Let strongly nonlinear EHL model problem of a ball rolling in the positive $x$-direction gives rise to a variational inequality
defined below as
\begin{align}\label{eq:1}
 \frac{\partial }{\partial x} \Big(\epsilon^{*} \frac {\partial u}{\partial x}\Big)+
 \frac{\partial }{\partial y} \Big(\epsilon^{*} \frac {\partial u}{\partial y}\Big)
 \le \frac {\partial (\rho h)}{\partial x}
\end{align}
\begin{align}\label{eq:2}
 u\ge 0
\end{align}
\begin{align}\label{eq:3}
 u.\Big[\frac{\partial }{\partial x} \Big(\epsilon^{*} \frac {\partial u}{\partial x}\Big)+
 \frac{\partial }{\partial y} \Big(\epsilon^{*} \frac {\partial u}{\partial y}\Big)
 -\frac {\partial (\rho h)}{\partial x}\Big] = 0,
\end{align}
Here term $\epsilon$  is defined as
\begin{equation*}
 \epsilon = \frac{\rho \mathscr{H}^{3}}{\eta\lambda},
\end{equation*}
where $\rho$ is dimensionless density of lubrication, $\eta$ is dimensionless viscosity of lubrication and
speed parameter
\begin{align}
  \label{eqn9}
   \lambda= \dfrac{6\eta_{0}u_{s}R^{2}}{a^{3}p_{H}}.
  \end{align}
The non-dimensionless viscosity $\eta$ is defined according to 
\begin{align}
\label{eqn10}
 \eta(u) = \exp\Bigg\{ \Bigg( \dfrac{\alpha p_{0}}{z}  \Bigg) 
 \Bigg(-1+\Big(1+\dfrac{{u}p_{H}}{p_{0}}\Big)^{z}   \Bigg)   \Bigg\}.
\end{align}
Dimensionless density $\rho$ is given by
\begin{align}
\label{eqn11}
 \rho(u) = \dfrac{0.59 \times 10^{9} + 1.34 u p_{H}}{0.59 \times 10^{9} + u p_{H}}.
\end{align}
The term film thickness $\mathscr{H}$ of lubricant is written as follows
\begin{align}
\label{eqn12}
\mathscr{H}(x,y) = \mathscr{H}_{00}+\frac{x^{2}}{2}+\frac{y^{2}}{2} + 
\frac{2}{\pi^{2}}\int_{-\infty}^{\infty} \int_{-\infty}^{\infty}\frac{u(x^{'},y^{'})dx^{'}dy^{'}}{\sqrt{(x-x^{'})^2+(y-y^{'})^2}},
\end{align}
where $\mathscr{H}_{00}$ is an integration constant.\\
The dimensionless force balance equation is defined as follows
\begin{gather}
\label{eqn13}
 \int_{-\infty}^{\infty} \int_{-\infty}^{\infty}u(x',y') dx'dy' = \frac{3\pi}{2}
\end{gather}
Consider the ball is elastic whenever load is large enough. Then system \ref{eq:1}--\ref{eq:6} and \ref{eq:1}--\ref{eq:6} 
form line and point contact Elasto-hydrodynamic Lubrication model respectively.
Schematic diagrams of EHL model is given in \ref{fig:undefm} and \ref{fig:defm} in the form of undeformed and deformed contacting body structure respectively.\\
The remainder of the article is organized as follows. In section \ref{section:vi} variational inequality and its notation is established;
Furthermore, existence results are proved for our model problem; In section.~\ref{section:dgfvm} DG-FEM notation 
and the proposed method is demonstrated; In section.~\ref{section:error} Error estimates are proved in $L^2$ and $H^{1}$ norm;
In section.~\ref{section:ntest} numerical experiment and graphical results are provided;
At last section.~\ref{section:con} conclusion and future direction is mentioned.
\section{Variational Inequality}\label{section:vi}
We consider space $\mathscr{V} = H^{1}_{0}(\Omega)$ and its dual space as $\mathscr{V}^{*}= (H^{1}_{0}(\Omega))^{*} = H^{-1}(\Omega)$.
Also define notion $\langle.,.\rangle$ as duality pairing on $\mathscr{V}^{*} \times \mathscr{V} $.
Further assume that $\mathscr{C}$ is closed convex subset of $\mathscr{V}$ defined by 
\begin{align}\label{eq:11}
\mathscr{C} = \Big\{  v \in \mathscr{V}: v \ge 0  \text{ a.e. } \in \Omega \Big\}.
\end{align}
Additionally, we define the operator $\mathscr{T}$ as 
\begin{align}\label{eq:12}
 \mathscr{T}:u\rightarrow -\Big[\frac{\partial }{\partial x} \Big(\epsilon^{*} \frac {\partial u}{\partial x}\Big)+
 \frac{\partial }{\partial y} \Big(\epsilon^{*} \frac {\partial u}{\partial y}\Big)\Big]
+\frac {\partial (\rho h_{d})}{\partial x}
\end{align}
Then, for a given $f \in \mathscr{V}^{*}$, the problem of finding an element $u \in \mathscr{C}$ such that
\begin{align}\label{eq:14}
 \langle \mathscr{T}(u)-f,v-u \rangle \ge 0, \quad \forall v \in \mathscr{C}.
\end{align}
Throughout the article, we shall assume that there exists $ \epsilon_{1},M_{*} \in \mathbb{R}_{+}$ such that
\begin{align}
0 < \epsilon_{1} \le \epsilon(u) \le M_{*} \quad \forall \varsigma \in \Omega \quad \text{and} \quad u \in \mathbb{R}.
\end{align}
\begin{definition}
 Operator $\mathscr{T}: \mathscr{C} \subset \mathscr{V}\rightarrow \mathscr{V}^{*}$ is said to be pseudo-monotone if 
 $\mathscr{T}$ is a bounded operator and whenever $u_{k}\rightharpoonup u$ in $\mathscr{V}$ as $k \rightarrow \infty$
 and 
 \begin{align}\label{eq:b14}
 \lim_{k \rightarrow \infty}\sup\langle \mathscr{T}(u_{k}),u_{k}-u\rangle \le 0.
\end{align}
it follows that for all $v \in \mathscr{C}$
 \begin{align}\label{eq:15}
 \lim_{k \rightarrow \infty}\inf\langle \mathscr{T}(u_{k}),u-v\rangle 
 \ge \langle\mathscr{T}(u),u-v\rangle.
\end{align}
\end{definition}
\begin{definition}
Operator $ \mathscr{T}: \mathscr{V} \rightarrow \mathscr{V}^{*}$ is said to be hemi-continuous 
if and only if the function $\phi: t \longmapsto \langle \mathscr{T}(tx+(1-t)y),x-y\rangle $ is
continuous on $[0,1] \quad \forall x,y \in \mathscr{V}$.
\end{definition}
On this context the following existence theorem has been proved by Oden and Wu \cite{oden1985} by assuming constant density 
and constant viscosity of the lubricant.
However, idea is easily extend-able for more realistic operating condition in which density 
and viscosity of the lubricant are depend on its applied pressure see Appendix.~\ref{appendix:pvalue}.
A straight forward modification of the analysis of \cite{oden1985} yields the theorem below and so we will omit the proof.
\begin{theorem}\cite{oden1985}
Let $\mathscr{C}(\neq \emptyset)$ be a closed, convex subset of a reflexive Banach space $\mathscr{V}$ and let $\mathscr{T}: \mathscr{C} \subset \mathscr{V} \rightarrow \mathscr{V}^{*}$ be a pseudo-monotone,
bounded, and coercive operator from $\mathscr{C}$ into the dual $\mathscr{V}^{*}$ of $\mathscr{V}$, in the sense that there exists $y \in \mathscr{C}$ such that
 \begin{align}\label{eq:18}
 \text{lim}_{||x||\rightarrow \infty}\frac{\langle  \mathscr{T}(x),x-y\rangle}{||x||} = \infty.
 \end{align}
Let $f$ be given in $\mathscr{V}^{*}$ then there exists at least one $u \in \mathscr{C}$ such that 
 \begin{align}\label{eq:19}
  \langle \mathscr{T}(x)-f,y-x\rangle \ge 0 \quad \forall y \in \mathscr{C}.
 \end{align}
\end{theorem}
In the next section, we will give a complete formulation as well as will give theoretical justification for existence of 
our model problem in discrete computed setting.
\section{Discrete Formulation of DG-FEM}\label{section:dgfvm}
\begin{figure}
\centering
\includegraphics[width=2.0in, height=2.0in, angle=0]{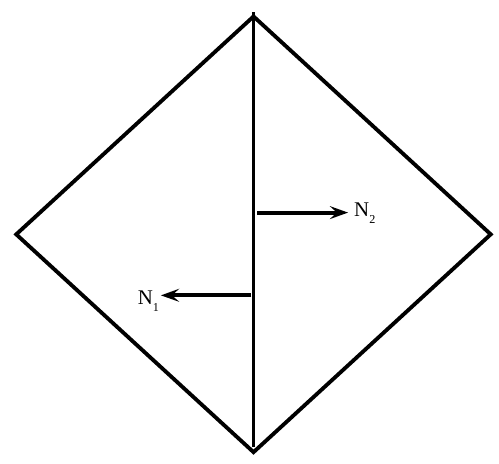}
\caption{Adjacent Element $K_{1}$ and $K_{2}$}
\label{fig:elm}
\end{figure}
Let $\mathcal{P}_{h}=\cup_{i \in \mathcal{J}_{h}}\{K_{i}\}$ is a discontinuous finite element partition of domain $\Omega$, where $\mathcal{J}_{h}:=\{i; 1\le i \le N_{h}\}$.
We define $\mathbb{P}_{p_{i}}(\overline{K})$ as the space of polynomials of total degree less than or equal $p_{i}$ on the master rectriangle $\overline{K}=[-1 ,1]\times [-1, 1]$.
Let $\overline{\mathcal{S}}_{p_{i}}(\overline{K})$ denote $\mathbb{P}_{p_{i}}(\overline{K})$ whenever $\overline{K}$ is a master rectriangle.
\begin{equation}
\overline{\mathcal{S}}^{p}_{h} (K_{i})= \{ v; v= \overline{v} \circ {\mathcal{M}^{-1}_{v}}, \overline{v} \in \overline{\mathcal{S}}_{p_{i}}(\overline{K})\}.
\end{equation}
We define finite dimensional discontinuous space as
\begin{equation}\label{eq:26}
\mathcal{D}^{p}_h = \{ v \in L^{2}(\Omega): v|_{K_{i}} \in \overline{\mathcal{S}}_{p_{i}}(K_{i}), v|_{\partial \Omega}=0 \quad \forall K_{i} \in
\mathcal{P}_{h} \},
\end{equation}
where $p= \min\{p_{i} \ge 1; 1 \le i \le N_{h}\}$.\\
Let $e_{k}$ be an interior edge shared by two elements $K_i$ and $K_j$ in $\mathcal{P}_h$ and let $\bold{N}_{i} $ 
and $\bold{N}_{j}$ be unit normal vectors on $e_{k}$ pointing exterior to $K_i$ and $K_j$ respectively.
We define average \{.\} and jump [.] on $e_{k}$ for scalar $q$ and vector $w$, respectively, as (\cite{arnold})
\[\{q\} = \frac{1}{2}(q|_{\partial K_i}+ q|_{\partial K_2}), \quad 
[q] = (q|_{\partial K_i}\bold{N}_{i}+ q|_{\partial K_j}\bold{N}_{j})\]
\[\{w\} = \frac{1}{2}(w|_{\partial K_i}+ w|_{\partial K_j}), \quad 
[w] = (w|_{\partial K_i}\bold{N}_{i}+ w|_{\partial K_j}\bold{N}_{j}).\]
If $e_{k}$ is a edge on the boundary of $\Omega$, we define ${q} = q,\quad [w] = w.\bold{N}$.
Let $\Gamma$ denote the union of the boundaries of the triangle $K$ of $\mathcal{P}_h$ and $\Gamma_{0}:= \Gamma\diagdown\partial \Omega$.
We define 
\begin{align}
 H^{s}(\Omega,\mathcal{P}_{h}):=\{ v \in L^{2}(\Omega): v|_{K_{i}} \in H^{s}(K_{i}), \quad \forall K_{i} \in \mathcal{P}_{h} \}.
\end{align}
Let $v \in H^{2}(\Omega,\mathcal{P}_{h})$, then we define the following mesh dependent norm $\vertiii{.}$ and $\vertiii{.}_{\nu}$ as
\begin{align}\label{eq:36}
\vertiii{v}^{2}= \sum_{i=1}^{N_{h}}\int_{K_{i}}|\nabla v|^{2}dx
+\sum_{e_{k} \in \Gamma} a_{k} \frac{p_{k}^{2}}{|e_{k}|^{\beta}} \int_{e_{k}}[v]^{2}\\
\vertiii{v}^{2}_{\nu}= \sum_{i=1}^{N_{h}}\int_{K_{i}}|\nabla v|^{2}dx
+\sum_{e_{k}}\frac{|e_{k}|^{\beta}}{p_{k}^{2}}\int_{e}\Big\{ \frac{\partial v }{\partial \nu}\Big\}^{2} ds 
+\sum_{e_{k} \in \Gamma} a_{k} \frac{p_{k}^{2}}{|e_{k}|^{\beta}} \int_{e_{k}}[v]^{2}.
\end{align}
\subsection{Exterior penalty solution approximation}
In this section, we introduce an exterior penalty term to regularize the inequality constraint \ref{eq:1}--\ref{eq:6}.
We define a exterior penalty operator ${\xi}:H_{0}^{1}(\Omega) \rightarrow H^{-1}$ as 
\begin{align}\label{eq:48}
 {\xi}(u) = u^{-},
\end{align}
where $u^{-} = u-\max(u,0)=\dfrac{u-|u|}{2}$.
Let us define exterior penalty problem,
$(\mathscr{U}_{\epsilon_{p}})$: for $\epsilon_{p} > 0,\quad \text{find}\quad u_{\epsilon_{p}} \in H^{1}_{0}(\Omega)$ such that
\begin{align}\label{eq:49}
 \langle \mathscr{T}(u_{\epsilon_{p}}), v \rangle + \langle {\xi}(u_{\epsilon_{p}}), 
 v \rangle/\epsilon_{p} = \langle f, v \rangle  \quad \forall v \in H^{1}_{0}(\Omega),
\end{align}
where ${\varepsilon}$ is an arbitrary small positive number (${\varepsilon} = 1.0 \times 10^{-6}$).
\begin{lemma}
Penalty operator $\xi: \mathscr{V}\longmapsto \mathscr{V^{*}}$ is monotone, coercive and bounded.
\end{lemma}
Now from regularity theory (see reference []), it is easy to show that there exists a unique solution $u_{\epsilon_{p}} \in H^{2}(\Omega)$ such that 
\begin{align}
|\vert u_{\epsilon_{p}}  |\vert_{H^{2}(\Omega)} \le C|\vert f_{\epsilon_{p}} |\vert_{0},
\end{align}
where $f_{\epsilon_{p}}=f- {\xi/\epsilon_{p}}(u_{\epsilon_{p}})$.
\subsection{Weak Formulation}
Reconsider the problem of the type 
\begin{align}
-\frac{\partial }{\partial x} \Big(\epsilon^{*}(u) \frac {\partial u}{\partial x}\Big)-
 \frac{\partial }{\partial y} \Big(\epsilon^{*}(u) \frac {\partial u}{\partial y}\Big)
 +\frac {\partial (\rho h_{d})}{\partial x}+\frac{1}{\epsilon_{p}}{\xi}(u)=0 \quad \text{in } \Omega\\
 u =0 \quad \text{on } \partial \Omega,
\end{align}
where all notation has their usual meaning.\\
For given $u,v \in H^{2}(\Omega, \mathcal{P}_{h})$ and for fixed value of $\Phi, h_{d} \in H^{2}(\Omega, \mathcal{P}_{h})$, define bilinear form as
\begin{align}\label{eq:33}
 \langle \mathscr{T}(\Phi;u),v\rangle 
 = \sum \int_{K_{i}} \epsilon^{*}(\Phi) \nabla u.\nabla v ds +\sum \int_{K_{i}} \frac{1}{\epsilon_{p}}{\xi}(\Phi) v ds \nonumber \\
- \sum\limits_{e_{k} \in \Gamma}\int_{e_{k}} [v]\{ \epsilon^{*}(\Phi) \nabla u.\bold{n}\}ds 
+\sum_{e_{k} \in \Gamma} a_{k} \frac{p_{k}^{2}}{|e_{k}|^{\beta}} \int_{e_{k}}[u][v] ds \nonumber \\
-\sum\limits_{K \in \mathscr{R}_h} \int_{K_{i}}
(\rho(\Phi) h_{d}(x)).(\beta .\bold{n})\nabla v ds 
+\sum\limits_{e_{k} \in \Gamma} \int_{e_{k}} [v] \{ (\rho(\Phi) h_{d}(x)).(\beta.\bold{n})\} ds.
\end{align}
Now we will state few lemmas and inequalities without proof which will be later helpful in our subsequent analysis.
\begin{lemma}[\bf Interpolation Error Estimates]
For $u \in H^{s}(K_{i})$, there exist a positive constant $C_{A}$ and an interpolation value $u_{I} \in \mathscr{V}_{h}$, such that
\begin{align}
 |\vert u-u_{I} \vert|_{s,K} \le C_{A}h^{2-s}|u|_{2,K},\quad s=0,1.
\end{align}
\end{lemma}
{\bf Trace inequality.} We state without proof the following trace inequality. Let $\phi \in H^{2}(K)$ and for an edge $e$ of $K$,
\begin{align}
 |\vert \phi |\vert_{e}^{2} \le C(h_{e}^{-1}|\phi|_{K}^{2}+h_{e}|\phi|_{1,K}^{2}).
\end{align}
Next lemma provides us a bound of film thickness term and later helpful in proving coercivity and error analysis.
\begin{lemma}
 For $h_d$ defined in equation \ref{eq:5}, $0 < \beta_{*} < 1, s= 2-\beta_{*}/(1-\beta_{*})>2$ there exist $C_{1} \text{ and } C_{2}>0$
 such that
 \begin{align}\label{eq:38}
  \max_{x,y \in \Omega}|h_{d}(u)| \le C_{1}+C_{2}\lVert u \rVert_{L^{s}} \quad 0 < \beta_{*} <1, \quad \forall (x,y) \in \bar{\Omega}.
 \end{align}
\end{lemma}
\begin{lemma}
 The operator $ \mathscr{T}$ defined in equation \ref{eq:34} is bounded as a map from $\mathscr{V}$ into $\mathscr{V}^{*}$.
\end{lemma}
\begin{lemma}
 The operator $\mathscr{T}$, defined in equation (21) is hemi-continuous, that is $\forall u, v,w \in \mathscr{V}$,
 \[ \lim_{t \rightarrow 0^{+}}\langle \mathscr{T}(u+tv), w \rangle=\langle \mathscr{T}(u), w \rangle.\]
\end{lemma}
\begin{lemma}
 The operator defined on equation (21) is coercive i.e. there is a constant $C$ independent of $h$ such that for $\alpha_{1}$ large enough and $h$ is small enough 
 \begin{align}\label{eq:47}
 \langle  \mathscr{T}(u;u_{h}),u_{h} \rangle \ge C \vertiii{u_{h}}^{2} \quad \forall u_{h} \in \mathscr{V}_{h}
 \end{align}
\end{lemma}
\subsection{Linearizion}
\begin{align}
-\nabla (\epsilon^{*}(u)\nabla \psi+\epsilon_{u}^{*}(u) \nabla u \psi)+\nabla \Big(\vec{\beta}(\rho h_{d}
+(\rho h_{d})_{u}\psi)\Big)= \phi_{h} \text{ in } \Omega \\
\psi=0 \text{ on } \partial \Omega.
\end{align}
and $\psi$ satisfying the elliptic regularity as
\begin{align}
 |\vert \psi |\vert_{H^{2}(\Omega)} \le C|\vert \phi_{h} |\vert_{0}.
\end{align}
We seek $u_{h} \in \mathcal{D}^{p}_h(\mathcal{P}_{h})$ such that 
\begin{align}
 \mathscr{B}(u;u,v_{h})=\mathscr{B}(u_{h};u_{h},v_{h})
\end{align}
Now from Taylor's series expansion we get
\begin{align}
 \epsilon^{*}(w)=\epsilon^{*}(u) + \tilde{\epsilon}^{*}_{u}(w)(w-u),
\end{align}
where $\tilde{\epsilon}^{*}_{u}(w)=\int_{0}^{1} \epsilon^{*}_{u}(w+t[u-w]) dt$ \\
and
\begin{align}
 \epsilon^{*}(w)=\epsilon^{*}(u) + \epsilon ^{*}_{u}(u)(w-u)+\tilde{\epsilon}^{*}_{uu}(w)(w-u)^{2},
\end{align}
where $\tilde{\epsilon}^{*}_{uu}(w)=\int_{0}^{1}(1-t) \epsilon^{*}_{uu}(w+t[w-u]) dt$.
Consider the following bilinear form $\tilde{\mathscr{B}}$ as
\begin{align}
\tilde{\mathscr{B}}(\psi;w,v) =\mathscr{B}(\psi;w,v) + \sum_{i=1}^{N_{h}}\int_{K_{i}}(\epsilon_{u}(\psi)\nabla \psi)w \nabla v 
-\sum_{e_{k} \in \Gamma_{I}}\int_{e_{k}}\Big\{\epsilon_{u}(\psi)\frac{\partial \psi}{\partial \nu}w\Big\}[v] \nonumber \\
-\sum_{i=1}^{N_{h}}\int_{K_{i}}(\vec{\beta}.\bold{n})(\rho h_{d})_{u}\psi \nabla v 
+ \sum_{e_{k} \in \Gamma_{I}}\int_{e_{k}}\Big\{(\rho h_{d})_{u}(\vec{\beta}.\bold{n})w \Big\}[v].
\end{align}
Note that $\tilde{\mathscr{B}}$ is linear in $w$ and $v \in H^{2}(\Omega,\mathcal{P}_{h})$ for fixed value $\psi$.
It is clear from the assumptions on $\epsilon (u)$ and (from existence lemma (state the precise lemma here) for nonlinear elliptic PDE ) we have 
a unique solution $\psi \in H^{2}({\Omega})$ to the following elliptic problem :
\begin{align}
 -\nabla (\epsilon(u)\nabla \psi + \epsilon_{u}(u)\nabla u \psi) + \nabla \vec{\beta}(\rho h_{d}+(\rho h_{d})_{u}\psi) =\phi_{h} \text{ in } \Omega. \\
 \psi= 0 \text{ on } \partial \Omega.
\end{align}
and $\psi$ satisfies the following elliptic regularity condition 
\begin{align}
 |\vert \psi |\vert_{H^{2}(\Omega)} \le C|\vert \phi_{h} |\vert_{0}.
\end{align}
Now we will linearize problem (?) around $\mathcal{I}_{h}u$ which will be helpful in deriving few estimates. Substracting $\mathscr{B}(u;u,v)$ from both side of 
equation (?) we get
\begin{align}
\mathscr{B}(u;e,v_{h})=\sum_{i=1}^{N_{h}}\int_{K_{i}}(\epsilon(u_{h})-\epsilon(u))\nabla u_{h}.\nabla v_{h}
-\sum_{e_{k} \in \Gamma_{I}}\int_{e_{k}}\Big\{(\epsilon(u_{h})-\epsilon(u)) \frac{\partial u_{h}}{\partial \nu}\Big\}[v_{h}] \nonumber \\
-\theta \sum_{e_{k} \in \Gamma_{I}}\int_{e_{k}}\Big\{(\epsilon(u_{h})-\epsilon(u)) \frac{\partial v_{h}}{\partial \nu}\Big\}[u_{h}]
+\sum_{i=1}^{N_{h}}\int_{K_{i}}\{ (\rho h_{d})(u_{h})-(\rho h_{d})(u)\}\vec{\beta}.\bold{n}\nabla v_{h} ds \nonumber \\
-\sum_{e_{k} \in \Gamma_{I}}\int_{e_{k}}\Big\{(\rho h_{d})(u_{h})-(\rho h_{d})(u)) \vec{\beta}.\bold{n}\Big\}[v_{h}] ds.
\end{align}
Since $[u]=0$ on each $e_{k} \in \Gamma_{I}$, we rewrite the equation as
\begin{gather}
 \mathscr{B}(u;e,v_{h})=\sum_{i=1}^{N_{h}}\int_{K_{i}}(\epsilon(u_{h})-\epsilon(u))\nabla (u_{h}-u).\nabla v_{h}
+\sum_{i=1}^{N_{h}}\int_{K_{i}}(\epsilon(u_{h})-\epsilon(u))\nabla u.\nabla v_{h} \nonumber \\ 
-\sum_{e_{k} \in \Gamma_{I}}\int_{e_{k}}\Big\{(\epsilon(u_{h})-\epsilon(u)) \frac{\partial (u_{h}-u)}{\partial \nu}\Big\}[v_{h}]
-\sum_{e_{k} \in \Gamma_{I}}\int_{e_{k}}\Big\{(\epsilon(u_{h})-\epsilon(u)) \frac{\partial u}{\partial \nu}\Big\}[v_{h}]\nonumber \\
-\theta \sum_{e_{k} \in \Gamma_{I}}\int_{e_{k}}\Big\{(\epsilon(u_{h})-\epsilon(u)) \frac{\partial v_{h}}{\partial \nu}\Big\}[u_{h}-u]
+\sum_{i=1}^{N_{h}}\int_{K_{i}}\{ (\rho h_{d})(u_{h})-(\rho h_{d})(u)\}\vec{\beta}.\bold{n}\nabla v_{h} ds \nonumber \\
-\sum_{e_{k} \in \Gamma_{I}}\int_{e_{k}}\Big\{(\rho h_{d})(u_{h})-(\rho h_{d})(u)) \vec{\beta}.\bold{n}\Big\}[v_{h}] ds.
\end{gather}
Now adding both side by 
\begin{gather}
-\sum_{i=1}^{N_{h}}\int_{K_{i}} \epsilon_{u}(u)(u_{h}-u)\nabla u. \nabla v_{h}
+\sum_{e_{k} \in \Gamma_{I}}\int_{e_{k}}\Big\{\epsilon_{u}(u)(u_{h}-u)\frac{\partial u}{\partial \nu}\Big\}[v_{h}]\nonumber \\
-\sum_{i=1}^{N_{h}}\int_{K_{i}}\{ (\rho h_{d})_{u}(u)(u_{h}-u)\}\vec{\beta}.\bold{n}\nabla v_{h} ds
+\sum_{e_{k} \in \Gamma_{I}}\int_{e_{k}}\Big\{(\rho h_{d})_{u}(u)(u_{h}-u) \vec{\beta}.\bold{n}\Big\}[v_{h}] ds.
\end{gather}
By writing $e=u-u_{h}=u-\mathcal{I}_{h}u+\mathcal{I}_{h}u-u_{h}$ and using Taylor's formulae equation (?) we rewrite the term as
\begin{gather}
 \tilde{\mathscr{B}}(u;\mathcal{I}_{h}u-u_{h},v_{h})=\tilde{\mathscr{B}}(u;\mathcal{I}_{h}u-u,v_{h})+\tilde{\mathscr{F}}(u_{h};u_{h}-u,v_{h}),
\end{gather}
where 
\begin{gather}
\tilde{\mathscr{F}}(u_{h};u_{h}-u,v_{h})
=\sum_{i=1}^{N_{h}}\int_{K_{i}}\tilde{\epsilon}_{u}(u_{h}) e \nabla e.\nabla v_{h}
+\sum_{i=1}^{N_{h}}\int_{K_{i}}\tilde{\epsilon}_{uu}(u_{h}) e^{2}\nabla u.\nabla v_{h} \nonumber \\
-\sum_{e_{k} \in \Gamma_{I}}\int_{e_{k}}\Big\{\tilde{\epsilon}_{uu}(u_{h}) e^{2}\frac{\partial u}{\partial \nu}\Big\}[v_{h}]
-\sum_{e_{k} \in \Gamma_{I}}\int_{e_{k}}\Big\{\tilde{\epsilon}_{u}(u_{h}) e\frac{\partial e}{\partial \nu}\Big\}[v_{h}]\nonumber \\
-\theta \sum_{e_{k} \in \Gamma_{I}}\int_{e_{k}}\Big\{\tilde{\epsilon}_{u}(u_{h}) e\frac{\partial v_{h}}{\partial \nu}\Big\}[e]
+\sum_{i=1}^{N_{h}}\int_{K_{i}}\{ \tilde{(\rho h_{d})}_{uu}(u_{h}) e^{2}\}\vec{\beta}.\bold{n}\nabla v_{h} ds \nonumber \\
-\sum_{e_{k} \in \Gamma_{I}}\int_{e_{k}}\Big\{\tilde{(\rho h_{d})}_{uu}(u_{h}) e^{2}\vec{\beta}.\bold{n}\Big\}[v_{h}] ds.
\end{gather}
\subsection{Existence and Uniqueness}
For a given $z \in \mathcal{D}^{p}_{h}(\mathcal{P}_{h})$, let $\mathcal{S}_{h}:\mathcal{D}^{p}_{h}(\mathcal{P}_{h}) \rightarrow  \mathcal{D}^{p}_{h}(\mathcal{P}_{h})$
be a mapping define as $y= \mathcal{S}_{h} z \in \mathcal{D}^{p}_{h}(\mathcal{P}_{h}) $ and satisfies
\begin{align}
\tilde{\mathscr{B}}(u;\mathcal{I}_{h}u-y,v_{h}) =\tilde{\mathscr{B}}(u;\mathcal{I}_{h}u-u,v_{h}) + \tilde{\mathscr{F}}(z;z-u,v_{h}) \forall v_{h} \in \mathcal{D}^{p}_{h}(\mathcal{P}_{h}). 
\end{align}
\begin{lemma}
Consider $\beta \ge 1$ and $z,v_{h} \in \mathcal{D}^{p}_{h}(\mathcal{P}_{h})$ also define $\chi= z-\Pi_{h} u$ and $\eta=u-\Pi_{h} u$, then there
exist a constant $C$ independent of h and p such that the following condition satisfies
\begin{align}
\Big|\mathcal{F}(z;z-u,v_{h})\Big| \le CC_{\alpha}\Big[ \Big(\max_{1\le i \le N_{h}}\frac{p_{i}}{h_{i}} \Big)^{1/2}\vertiii{\chi}^{2}
+C_{u}h^{1/2}(\vertiii{\chi}+\vertiii{\eta})\Big]\vertiii{v_{h}}.
\end{align}

\end{lemma}
\begin{proof}
Let $z \in \mathcal{D}^{p}_{h}(\mathcal{P}_{h})$ and set $\vartheta=z-u$. Now consider equation (4.13) and substitute $u_{h}$ by $z$ and by $z-u$
to get
\begin{align}
\mathcal{F}(z;\vartheta,v_{h})=\sum_{i=1}^{N_{h}}\int_{K_{i}} \tilde{\epsilon}_{u}(z)\vartheta \nabla \vartheta.\nabla v_{h}
+\sum_{i=1}^{N_{h}}\int_{K_{i}} \tilde{\epsilon}_{uu}(z)\vartheta^{2} \nabla u.\nabla v_{h} \nonumber \\
-\sum_{e_{k} \in \Gamma_{I}}\int_{e_{k}} \Big\{\tilde{\epsilon}_{uu}(z)\vartheta^{2}\nabla u.\bold{n}\Big\}[v_{h}]
-\sum_{e_{k} \in \Gamma_{I}}\int_{e_{k}} \Big\{\tilde{\epsilon}_{u}(z)\vartheta\nabla \vartheta.\bold{n}\Big\}[v_{h}] \nonumber \\
-\theta \sum_{e_{k} \in \Gamma_{I}}\int_{e_{k}} \Big\{\tilde{\epsilon}_{u}(z)\vartheta\nabla v_{h}.\bold{n}\Big\}[\vartheta]
+\sum_{i=1}^{N_{h}}\int_{K_{i}}\tilde{(\rho h_{d})}_{uu}(z) \vartheta^{2} \vec{\beta}.\bold{n} \nabla v_{h}\nonumber \\
-\sum_{e_{k} \in \Gamma_{I}}\int_{e_{k}} \Big\{\tilde{(\rho h_{d})}_{uu}(z)\vartheta^{2} \vec{\beta}.\bold{n}\Big\}[v_{h}].
\end{align}
Now putting the value of $\vartheta=\chi-\eta$, where $\chi=z-\Pi_{h}u$ and $\eta=u-\Pi_{h}u$ in above equation (35) and estimating 
the right hand side term we get
First part of the right hand side of equation (35) is approximated as
\begin{align}
 |I| = \Big|\sum_{i=1}^{K_{h}}\int_{K_{i}}\tilde{\epsilon}_{u}(z)\vartheta \nabla \vartheta.\nabla v_{h} \Big| 
 \le \Big|\sum_{i=1}^{K_{h}}\int_{K_{i}}\tilde{\epsilon}_{u}(z)\chi \nabla \chi.\nabla v_{h} \Big| \nonumber \\
 \Big|\sum_{i=1}^{K_{h}}\int_{K_{i}}\tilde{\epsilon}_{u}(z)\chi \nabla \eta.\nabla v_{h} \Big|
 +\Big|\sum_{i=1}^{K_{h}}\int_{K_{i}}\tilde{\epsilon}_{u}(z)\eta \nabla \chi.\nabla v_{h} \Big|+
 \Big|\sum_{i=1}^{K_{h}}\int_{K_{i}}\tilde{\epsilon}_{u}(z)\eta \nabla \eta.\nabla v_{h} \Big|.
\end{align}
Now using inverse inequality estimates stated below without proof as
\begin{lemma}[Inverse Inequality]
Suppose $v_{h} \in \mathcal{Z}(K_{i})$ and let $r \ge 2$. Then $\exists C >0$ such that following conditions holds
\begin{align}
 \lVert v_{h} \rVert_{L^{r}(K_{i})} \le C_{I}p_{i}^{1-2/r}h_{i}^{2/r-1}\lVert v_{h} \rVert_{L^{2}(K_{i})} \\
 | v_{h} |_{H^{l}(K_{i})} \le C_{I}p_{i}^{2}h_{i}^{-1}|v_{h}|_{H^{l-1}(K_{i})} \\
 \lVert v_{h} \rVert_{L^{r}(e_{k})} \le C_{I}p_{i}^{1-2/r}|e_{k}|^{1/r-1/2}\lVert v_{h} \rVert_{L^{2}(e_{k})}.
\end{align}
\end{lemma}
Now using Holder's inequality, first part of right hand side of equation (36) is estimated as
\begin{align}
\Big|\sum_{i=1}^{K_{h}}\int_{K_{i}}\tilde{\epsilon}_{u}(z)\chi \nabla \chi.\nabla v_{h} \Big|
\le C_{\epsilon}\sum_{i=1}^{N_{h}}\lVert \chi \rVert_{L^{6}(K_{i})}\lVert \nabla \chi \rVert_{L^{3}(K_{i})}
\lVert \nabla v_{h} \rVert_{L^{2}(K_{i})}
\end{align}
Now using the inverse property of lemma 3.8 we have
\begin{align}
\Big|\sum_{i=1}^{K_{h}}\int_{K_{i}}\tilde{\epsilon}_{u}(z)\chi \nabla \chi.\nabla v_{h} \Big|
\le C_{\epsilon}\Big(\max_{1 \le i \le N_{h}}\frac{p_{i}}{h_{i}}\Big)^{1/3}\vertiii{\chi}^{2}\vertiii{v_{h}}
\end{align}
Similarly, second part of right hand side of equation (36) is approximated as 
\begin{align}
 \Big|\sum_{i=1}^{K_{h}}\int_{K_{i}}\tilde{\epsilon}_{u}(z)\chi \nabla \eta.\nabla v_{h} \Big| \le
 C_{\epsilon}\sum_{i=1}^{N_{h}}\lVert \chi \rVert_{L^{6}(K_{i})}\lVert \nabla \eta \rVert_{L^{3}(K_{i})}
\lVert \nabla v_{h} \rVert_{L^{2}(K_{i})}.
\end{align}
Now using above property of inverse inequality in equation (38) we obtain
\begin{align}
 \Big|\sum_{i=1}^{K_{h}}\int_{K_{i}}\tilde{\epsilon}_{u}(z)\chi \nabla \eta.\nabla v_{h} \Big| \le
 C_{\epsilon}\sum_{i=1}^{N_{h}}\lVert \chi \rVert_{L^{6}(K_{i})}\lVert \nabla \eta \rVert_{L^{3}(K_{i})}
\lVert \nabla v_{h} \rVert_{L^{2}(K_{i})} \nonumber \\
\le  C_{\epsilon}\frac{h^{2/3}}{p^{2/3}} \lVert u \rVert_{H^{2}(\Omega)}\vertiii{\chi} \vertiii{v_{h}}
\end{align}
Now using similar property we can show that 
\begin{align}
 \Big|\sum_{i=1}^{K_{h}}\int_{K_{i}}\tilde{\epsilon}_{u}(z)\eta \nabla \chi.\nabla v_{h} \Big|
\le  C_{\epsilon}\frac{h^{2/3}}{p^{2/3}} \lVert u \rVert_{H^{2}(\Omega)}\vertiii{\chi} \vertiii{v_{h}}
\end{align}
and 
\begin{align}
 \Big|\sum_{i=1}^{K_{h}}\int_{K_{i}}\tilde{\epsilon}_{u}(z)\eta \nabla \eta.\nabla v_{h} \Big|
\le  C_{\epsilon}\frac{h^{2/3}}{p^{2/3}} \lVert u \rVert_{H^{2}(\Omega)}\vertiii{\eta} \vertiii{v_{h}}
\end{align}
hold.\\
Now second term of right hand side of equation $(35)$ is estimated as 
\begin{align}
 \Big|\sum_{i=1}^{N_{h}}\int_{K_{i}}\tilde{\epsilon_{uu}(z)}\vartheta^{2} \nabla u. \nabla v_{h} \Big| \le 
 C_{\epsilon}\sum_{i=1}^{N_{h}}\int_{K_{i}}|\chi^{2} \nabla u. \nabla v_{h}|
 +C_{\epsilon}\sum_{i=1}^{N_{h}}\int_{K_{i}}|\eta^{2} \nabla u. \nabla v_{h}| \nonumber \\
 +2C_{\epsilon}\sum_{i=1}^{N_{h}}\int_{K_{i}}|\chi.\eta \nabla u. \nabla v_{h}|.
\end{align}
Now using Holder's inequality in right hand side of above equation (46) we have 
\begin{align}
 \sum_{i=1}^{N_{h}}\int_{K_{i}}|\chi^{2} \nabla u. \nabla v_{h}| \le
 \Big(\max_{1 \le i \le N_{h}}\frac{p_{i}}{h_{i}}\Big)^{1/3}\vertiii{\chi}^{2}\vertiii{v_{h}}.
\end{align}
Second part of right hand side of equation (46) is estimated as
\begin{align}
 \sum_{i=1}^{N_{h}}\int_{K_{i}}|\eta^{2} \nabla u. \nabla v_{h}| \le
  h^{3/2}\lVert u \rVert_{H^{1}(\Omega)}|u |_{W^{1}_{\infty}(\Omega)}\vertiii{\eta}\vertiii{v_{h}}.
\end{align}
Third part of right hand side of equation (46) is estimated as 
\begin{align}
 \sum_{i=1}^{N_{h}}\int_{K_{i}}|\eta \chi \nabla u. \nabla v_{h}| \le
  h^{3/2}\lVert u \rVert_{H^{1}(\Omega)}|u |_{W^{1}_{\infty}(\Omega)}\vertiii{\chi}\vertiii{v_{h}}.
\end{align}
Now putting values from equation (46)-(48) in equation (45) we get the following estimate
\begin{align}
 \Big|\sum_{i=1}^{N_{h}}\int_{K_{i}}\tilde{\epsilon_{uu}(z)}\vartheta^{2} \nabla u. \nabla v_{h} \Big| \le 
 C_{\epsilon}\Big(\max_{1 \le i \le N_{h}}\frac{p_{i}}{h_{i}}\Big)^{1/3}\vertiii{\chi}^{2}\vertiii{v_{h}}+ \\
 C_{\epsilon}h^{3/2}\lVert u \rVert_{H^{1}(\Omega)}|u |_{W^{1}_{\infty}(\Omega)}\vertiii{\chi}\vertiii{v_{h}}
 +C_{\epsilon}h^{3/2}\lVert u \rVert_{H^{1}(\Omega)}|u |_{W^{1}_{\infty}(\Omega)}\vertiii{\eta}\vertiii{v_{h}}.
\end{align}
Now Third term of right hand side of equation (35) is estimated as
\begin{align}
 \Big|\sum_{e_{k} \in \Gamma_{I}}\int_{e_{k}} \Big\{\tilde{\epsilon}_{uu}(z)\vartheta^{2}\nabla u.\bold{n}\Big\}[v_{h}]\Big|
 \le C_{\epsilon}\sum_{e_{k} \in \Gamma_{I}}\int_{e_{k}}\Big| \Big\{\chi^{2}\nabla u.\bold{n}\Big\}\Big|\Big|[v_{h}]\Big| \nonumber \\
+C_{\epsilon}\sum_{e_{k} \in \Gamma_{I}}\int_{e_{k}}\Big| \Big\{\eta^{2}\nabla u.\bold{n}\Big\}\Big|\Big|[v_{h}]\Big|
+ 2 C_{\epsilon}\sum_{e_{k} \in \Gamma_{I}}\int_{e_{k}}\Big| \Big\{\eta \chi \nabla u.\bold{n}\Big\}\Big|\Big|[v_{h}]\Big| \nonumber \\
\le \nonumber \\
C_{\epsilon} \Big(\max_{1 \le i \le N_{h}} \frac{p_{i}^{1/2}}{h_{i}^{1-\beta/2}} \Big) \vertiii{\chi}^{2} 
\lVert u \rVert_{H^{1}(\Omega)}\vertiii{v_{h}}+ \nonumber \\
h^{(\beta + 1)/2}(1+ p)^{1/4}\vertiii{\eta}\vertiii{v_{h}}\lVert u \rVert_{H^{1}(\Omega)} \lVert u \rVert_{H^{2}(\Omega)} + \nonumber \\
h^{(\beta + 1)/2}(1+ p)^{1/4}\vertiii{\chi}\vertiii{v_{h}}\lVert u \rVert_{H^{1}(\Omega)} \lVert u \rVert_{H^{2}(\Omega)}. 
\end{align}
Now Fourth term of right hand side of equation (35) is estimated as
\begin{align}
 \Big| \sum_{e_{k} \in \Gamma}\int_{e_{k}}\Big\{\tilde{\epsilon}_{u}(z)\vartheta \nabla \vartheta .\bold{n}\Big\}[v_{h}]\Big|
 \le \sum_{e_{k} \in \Gamma}\int_{e_{k}}\Big|\Big\{\chi \nabla \chi .\bold{n}\Big\}\Big| \Big|[v_{h}]\Big|+ \nonumber \\
 \sum_{e_{k} \in \Gamma}\int_{e_{k}}\Big|\Big\{\eta \nabla \chi .\bold{n}\Big\}\Big| \Big|[v_{h}]\Big|+
 \sum_{e_{k} \in \Gamma}\int_{e_{k}}\Big|\Big\{\chi \nabla \eta .\bold{n}\Big\}\Big| \Big|[v_{h}]\Big|+ \nonumber \\
 \sum_{e_{k} \in \Gamma}\int_{e_{k}}\Big|\Big\{\eta \nabla \eta .\bold{n}\Big\}\Big| \Big|[v_{h}]\Big|.
\end{align}
Now Fifth term of right hand side of equation (35) is estimated as
\begin{align}
 \Big| \theta \sum_{e_{k} \in \Gamma_{I}}\int_{e_{k}}\Big\{\tilde{\epsilon}_{u}(z)
 \vartheta \nabla v_{h}.\bold{n}\Big\}[\vartheta] \Big| \le
 C_{\epsilon} \Big( \sum_{e_{k} \in \Gamma_{I}}\int_{e_{k}} \Big|\Big\{\eta \nabla v_{h}.\bold{n}\Big\}[\eta] \Big| \nonumber \\
 + \sum_{e_{k} \in \Gamma_{I}}\int_{e_{k}} \Big|\Big\{\eta \nabla v_{h}.\bold{n}\Big\}[\chi]\Big|
 + \sum_{e_{k} \in \Gamma_{I}}\int_{e_{k}}\Big| \Big\{\chi \nabla v_{h}.\bold{n}\Big\}[\eta]\Big| \nonumber \\
 + \sum_{e_{k} \in \Gamma_{I}}\int_{e_{k}} \Big|\Big\{\chi \nabla v_{h}.\bold{n}\Big\}[\chi]\Big|\Big)
\end{align}
Now Sixth part of right hand side of equation (35) is estimated as
\begin{align}
\Big|\sum_{i=1}^{N_{h}}\int_{K_{i}}\tilde{(\rho h_{d})}_{uu}(z) \vartheta^{2} \vec{\beta}.\bold{n} \nabla v_{h}\Big|
\le C_{\rho h}\Big(\vertiii{\chi}^{2}\vertiii{v_{h}}+(\frac{h}{p})^{2}\vertiii{\eta}\vertiii{v_{h}}\lVert u \rVert_{H^{2}(\Omega)}
\nonumber \\
+(\frac{h}{p})^{2}\vertiii{\chi}\vertiii{v_{h}}\lVert u \rVert_{H^{2}(\Omega)}\Big).
\end{align}
At last Seventh part of right hand side of equation (35) is estimated as
\begin{align}
\Big|\sum_{e_{k} \in \Gamma_{I}}\int_{e_{k}} \Big\{\tilde{(\rho h_{d})}_{uu}(z)\vartheta^{2} \vec{\beta}.\bold{n}\Big\}[v_{h}]\Big|
\le \vertiii{\chi}^{2}\vertiii{v_{h}}
+h^{(\beta + 1)/2}(1+ p)^{1/4} \nonumber \\
\vertiii{\chi}\vertiii{v_{h}}\lVert u \rVert_{H^{1}(\Omega)} \lVert u \rVert_{H^{2}(\Omega)}
+h^{(\beta + 1)/2}(1+ p)^{1/4}\vertiii{\eta}\vertiii{v_{h}}\lVert u \rVert_{H^{1}(\Omega)} \lVert u \rVert_{H^{2}(\Omega)}
\end{align}

\end{proof}
\begin{lemma}
Suppose $z \in \mathcal{D}^{p}_{h}(\mathcal{P}_{h})$ and $\beta \ge 1$. Also take $y=\mathcal{S}z$. Then there exists a 
non negative constant $C$ which is independent of $h$ and $p$ such that following condition
\begin{align}
 \vertiii{\Pi_{h} u-y} \le CC_{\epsilon}\Big[\Big(\max_{1 \le i \le N_{h}} \frac{p}{h} \Big)^{1/2} \vertiii{\Pi_{h} u - z}^{2}
 + C_{} \vertiii{\Pi_{h} u - z} \nonumber \\
 +CC_{\epsilon}(1+C_{\epsilon} h^{1/2}) \vertiii{\Pi_{h} u -u} \Big]
\end{align}
hold.
\end{lemma}
\begin{proof}
Consider $\chi= \Pi_{h}u-z$, $\eta=\Pi_{h}u-u$ and $\xi=\Pi_{h}u-y$. Take $v_{h}=\xi$ in (**)
\begin{align}
 \Big|\tilde{\mathcal{B}}(u;\eta,\xi)\Big| \le 
 C \vertiii{\eta}\vertiii{\xi}.
\end{align}
Put $v_{h}=\xi$ in lemma (**) to obtain 
\begin{align}
 \Big|\mathcal{F}(z;z-u,\xi)\Big| \le CC_{\epsilon}\Big(\max_{1 \le i \le N_{h}}\frac{p_{i}}{h_{i}}\Big)^{1/2}
 \vertiii{\chi}^{2}\vertiii{\xi}+CC_{\epsilon}C_{u}h^{1/2}(\vertiii{\chi}+\vertiii{\eta})\vertiii{\xi}.
\end{align}
Now from above equation (57) and equation (58) and also from equation (**) we obtain
\begin{align}
 \Big|\mathcal{F}(z;z-u,\xi)\Big| \le CC_{\epsilon}\Big(\Big(\max_{1 \le i \le N_{h}}\frac{p_{i}}{h_{i}}\Big)^{1/2} 
 \vertiii{\chi}^{2}\nonumber \\
 +CC_{\epsilon}C_{u}h^{1/2}(\vertiii{\chi}+(CC_{\epsilon}C_{u}h^{1/2}+1)\vertiii{\eta})\Big)\vertiii{\xi}. 
\end{align}
Now using the coericivity property we have 
\begin{align}
 \vertiii{\xi}^{2} \le CC_{\epsilon}\Big(\Big(\max_{1 \le i \le N_{h}}\frac{p_{i}}{h_{i}}\Big)^{1/2} 
 \vertiii{\chi}^{2}\nonumber \\
 +CC_{\epsilon}C_{u}h^{1/2}(\vertiii{\chi}+(CC_{\epsilon}C_{u}h^{1/2}+1)\vertiii{\eta})\Big)\vertiii{\xi}.
\end{align}
Hence we have the desire result.

\end{proof}
\subsection{$L^{2}$-Error Estimates}
In order to bound error $\lVert u-u_{h}\rVert_{L^{2}}$ in $L^{2}$ norms we use  well known Aubin-Nitsche duality argument.
\begin{theorem}
 Let $\epsilon \in C^{2}_{b}(\Omega \times \mathbb{R})$ and $u \in W^{1}_{\infty}(\Omega)$. Suppose $\mathcal{P}_{h}^{p}$ is a regular partition.
 Then for sufficiently small $h$, there exists a constant $C=C(\alpha_{*}, M)$ which independent of $h$ and $p$ such that 
 \begin{align}
  \lVert u-u_{h}\rVert_{L^{2}} \le C C_{*}\frac{h^{\mu}}{p^{s}}\lVert u\rVert_{H^{s}(\Omega)}.
 \end{align}
\end{theorem}
\begin{proof}
Consider the following adjoint problem 
\begin{align}
 -\nabla (\epsilon^{*}(u)\nabla \psi)+\epsilon_{u}^{*}(u) \nabla u . \nabla \psi - \vec{\beta} \Big(\rho h_{d}
+(\rho h_{d})_{u}\Big)\nabla\psi= e \text{ in } \Omega \\
\psi=0 \text{ on } \partial \Omega.
\end{align}
Now from elliptic regularity property, there exist a unique $\psi \in H^{2}(\Omega)$ which satisfies above linear elliptic problem (?) and
 \begin{align}
  \lVert \psi \rVert_{H^{2}(\Omega)} \le C \lVert e \rVert_{L^{2}(\Omega)}
 \end{align}
\end{proof}
holds.\\
By short computation, It is easy to show that
\begin{align}
 \lVert e \rVert ^{2}= \mathscr{B}(u;u,\psi)-\mathscr{B}(u_{h};u_{h},\psi ) + \sum_{i=1}^{N_{h}} \int_{K_{i}}\Big(e \tilde{\epsilon_{u}} \nabla e - \tilde{\epsilon_{uu}}e^{2}\nabla u \Big)\nabla \psi \nonumber \\
 -\sum_{e_{k} \in \Gamma_{I}} \int_{e_{k}}\Big( \Big\{ \tilde{\epsilon_{u}}e\frac{\partial e}{\partial \nu}\Big\} - \Big\{ \tilde{\epsilon_{uu}}e^{2}\frac{\partial u}{\partial \nu}\Big\} \Big)[\psi]
 -\theta\sum_{e_{k} \in \Gamma_{I}} \int_{e_{k}}\Big\{ \tilde{\epsilon_{u}}e\frac{\partial \psi}{\partial \nu}\Big\}[e] \nonumber \\ 
 +\sum_{i=1}^{N_{h}}\int_{K_{i}}\{ \tilde{(\rho h_{d})}_{uu} e^{2}\}\vec{\beta}.\bold{n}\nabla \psi 
-\sum_{e_{k} \in \Gamma_{I}}\int_{e_{k}}\Big\{\tilde{(\rho h_{d})}_{uu} e^{2}\vec{\beta}.\bold{n}\Big\}[\psi].
\end{align}
The first term on the right hand side of (84) is revised as
\begin{align}
I=\mathscr{B}(u;u,\psi)-\mathscr{B}(u_{h};u,\psi) + \mathscr{B}(u_{h};u,\psi)  -\mathscr{B}(u_{h};u_{h},\psi ) \nonumber \\
=\mathscr{B}(u;u,\psi-\chi)-\mathscr{B}(u_{h};u,\psi-\chi) + \mathscr{B}(u_{h};u,\psi-\chi)  -\mathscr{B}(u_{h};u_{h},\psi-\chi ),
\end{align}
where $\chi=\mathcal{I}_{h}^{*}\psi$ such that $\chi|_{\partial \Omega}=0$.
\begin{align}
I=  \sum_{i=1}^{N_{h}} \int_{K_{i}}(\epsilon^{*}(u)-\epsilon^{*}(u_{h}) \nabla u \nabla ( \psi-\chi)
 -\int_{K_{i}}\Big((\rho h_{d})(u)-(\rho h_{d})(u_{h})\Big)\nabla ( \psi-\chi)\nonumber \\
-\sum_{i=1}^{N_{h}} \int_{K_{i}}(\epsilon^{*}(u)-\epsilon^{*}(u_{h}) \nabla (u-u_{h}) \nabla ( \psi-\chi) 
+\sum_{i=1}^{N_{h}} \int_{K_{i}}\epsilon^{*}(u)\nabla (u-u_{h}) \nabla ( \psi-\chi) 
\end{align}
By using Cauchy-Schwarz inequality, we can bound first, second and fourth terms on the right hand side of equation $(86)$ as
\begin{align}
\Big |\sum_{i=1}^{N_{h}} \int_{K_{i}}(\epsilon^{*}(u)-\epsilon^{*}(u_{h}) \nabla u \nabla ( \psi-\chi) \Big | \le C_{\epsilon}\vertiii{e} \lVert \psi-\chi \rVert_{H^{1}(\Omega)}\nonumber \\
\le C_{\epsilon}\frac{h}{p}\vertiii{e}\lVert \psi \rVert_{H^{2}(\Omega)},
\end{align}
\begin{align}
\Big |\sum_{i=1}^{N_{h}} \int_{K_{i}}\Big((\rho h_{d})(u)-(\rho h_{d})(u_{h})\Big)\nabla ( \psi-\chi) \Big |\le C_{\rho}\frac{h}{p}\vertiii{e}\lVert \psi \rVert_{H^{2}(\Omega)}
\end{align}
and
\begin{align}
\Big |\sum_{i=1}^{N_{h}} \int_{K_{i}}\epsilon^{*}(u) \nabla (u-u_{h}) \nabla ( \psi-\chi) \Big | \le C_{\epsilon}\frac{h}{p}\vertiii{e} \lVert \psi \rVert_{H^{2}(\Omega)}
\end{align}
The third term of right hand side of equation $(86)$ is estimated by using H\"{o}lder's inequality as
\begin{align}
\Big |\sum_{i=1}^{N_{h}} \int_{K_{i}}(\epsilon^{*}(u)-\epsilon^{*}(u_{h}) \nabla (u-u_{h}) \nabla ( \psi-\chi)\Big |\le C\lVert e \rVert_{L^{3}(\Omega)}\vertiii{e}\lVert \psi-\chi \rVert_{W^{1}_{6}(\Omega)}\nonumber \\
\le C\vertiii{e}^{2}\lVert \psi \rVert_{H^{2}(\Omega)}.
\end{align}
Now the second term of right hand side of equation (84) is bounded using H\"{o}lder's inequality as
\begin{align}
\Big | \sum_{i=1}^{N_{h}} \int_{K_{i}}\Big(e \tilde{\epsilon_{u}} \nabla e - \tilde{\epsilon_{uu}}e^{2}\nabla u \Big)\nabla \psi \Big |
\le \Big | \sum_{i=1}^{N_{h}} \int_{K_{i}}e \tilde{\epsilon_{u}} \nabla e \nabla \psi \Big |+ \Big |\sum_{i=1}^{N_{h}} \int_{K_{i}}\tilde{\epsilon_{uu}}e^{2}\nabla u \nabla \psi   \Big |\nonumber \\
\le C \vertiii{e}^{2} \lVert \psi \rVert_{H^{2}(\Omega)}
\end{align}
Now the third term of right hand side of equation (84) is estimated as
\begin{align}
 \Big|\sum_{e_{k} \in \Gamma_{I}} \int_{e_{k}}\Big( \Big\{ \tilde{\epsilon_{u}}e\frac{\partial e}{\partial \nu}\Big\} 
 - \Big\{ \tilde{\epsilon_{uu}}e^{2}\frac{\partial u}{\partial \nu}\Big\} \Big)[\psi]\Big|
 \le \nonumber \\
\Big|\sum_{e_{k} \in \Gamma_{I}} \int_{e_{k}}\Big\{ \tilde{\epsilon_{u}}e\frac{\partial e}{\partial \nu}\Big\}[\psi]\Big|
 +\Big|\sum_{e_{k} \in \Gamma_{I}} \int_{e_{k}}\Big\{ \tilde{\epsilon_{uu}}e^{2}\frac{\partial u}{\partial \nu}\Big\}[\psi]\Big|
\end{align}

\section{Conclusion}\label{section:con}
In this article, we have discuss and analyze interior-exterior penalty base discontinuous Galerkin finite element method for solving EHL line as well as point contact problems. Convergence of discrete DG solution is proved using Brouwer's fixed point theorem. We have shown that optimal order of convergence in $L^{2}$ and $H^{1}$ norms is achieved in mesh size $h$ theoretically. However, suboptimal order convergence is achieved in polynomial degree $p$.
\section*{Acknowledgment}
This work is fully funded by DST-SERB Project reference no.PDF/2017/000202 under N-PDF fellowship program 
and working group at the Tata Institute of Fundamental Research, TIFR-CAM, Bangalore. 
\bibliographystyle{plain}
\bibliography{ref}

\begin{thebibliography}{1}

\bibitem{arnold}
D.~N. Arnold.
\newblock An interior penalty finite element method with discontinuous
  elements.
\newblock {\em SIAM J. Numer. Anal.}, 15(1):742--760, 1982.

\bibitem{gudi07}
T.~Gudi and A.~K. Pani.
\newblock Discontinuous galerkin methods for quasi-linear elliptic problems of
  nonmonotone type.
\newblock {\em SIAM J. Numer. Anal.}, 45(1):163--192, 2007.

\bibitem{oden1985}
Oden J.~T and S.~R. Wu.
\newblock Existence of solutions to the reynolds equation of elastohydrodynamic
  lubrication.
\newblock {\em Int. J. Engng Sci.}, 23(2):207--215, 1985.

\bibitem{Scholz}
R.~Scholz.
\newblock Numerical solution of the obstacle problem by the penalty method.
\newblock {\em Computing}, 32(1):297--306, 1984.

\bibitem{peeyush2020}
Peeyush Singh and Prawal Sinha.
\newblock Interior-exterior penalty approach for solving elasto-hydrodynamic
  lubrication problem: Part {I}.
\newblock {\em Int. Jour. of Numeri. Anal. and Modeling.}, 17(5):695--731,
  2020.

\bibitem{gudi08}
N.~Nataraj T.~Gudi and A.~K. Pani.
\newblock hp-discontinuous galerkin methods for strongly nonlinear elliptic
  boundary value problems.
\newblock {\em Numerische Mathematik}.

\end{thebibliography}
\end{document}